\documentclass[11pt,a4paper]{amsart}
\usepackage{amssymb,amsmath}

\numberwithin{equation}{section}
\newtheorem{theorem}{Theorem}[section]
\newtheorem{lemma}[theorem]{Lemma}
\newtheorem{proposition}[theorem]{Proposition}
\newtheorem{corollary}[theorem]{Corollary}

\theoremstyle{definition}
\newtheorem{definition}[theorem]{Definition}
\newtheorem{example}[theorem]{Example}

\newtheorem{remark}[theorem]{Remark}

\newcommand\Supp{\operatorname{Supp}}
\newcommand\Ass{\operatorname{Ass}}

\newcommand\Ann{\operatorname{Ann}}

\newcommand\Hom{\operatorname{Hom}}

\newcommand\Ext{\operatorname{Ext}}
\newcommand\Rad{\operatorname{Rad}}

\newcommand\grade{\operatorname{grade}}

\newcommand\height{\operatorname{height}}
\newcommand\Spec{\operatorname{Spec}}
\newcommand{\gam}{\Gamma_{I}}

\newcommand{\qism}{\stackrel{\sim}{\longrightarrow}}

\begin{document}

\author[P.~Schenzel]{Peter Schenzel}
\title[Connectedness and indecomposibility]
{On connectedness and indecomposibility of local cohomology modules}
\address{Martin-Luther-Universit\"at Halle-Wittenberg,
Institut f\"ur Informatik, D --- 06 099 Halle (Saale), Germany}
\email{peter.schenzel@informatik.uni-halle.de}


\begin{abstract} Let $I$ denote an ideal of a local Gorenstein ring $(R, \mathfrak m)$.
Then we show that the local cohomology module $H^c_I(R), c = \height I,$ is indecomposable
if and only if $V(I_d)$ is connected in codimension one. Here $I_d$ denotes the intersection of the highest dimensional primary components of $I.$ This is a partial extension of a result
shown by Hochster and Huneke in the case $I$ the maximal ideal. Moreover there is an
analysis of connectedness properties in relation to various aspects of local cohomology. Among
others we show that the endomorphism ring of $H^c_I(R)$ is a local Noetherian ring if $\dim R/I = 1.$
\end{abstract}

\maketitle

\section{Introduction}
Let $(A, \mathfrak m, k)$ denote a local Noetherian ring. Let $I \subset A$
be an ideal. Let $H^i_I(A), i \in \mathbb Z,$ denote the local cohomology
modules of $A$ with respect to $I.$ In their paper (cf. \cite{HH}) Hochster
und Huneke have shown -- among others -- that $H^d_{\mathfrak m}(A), d =
\dim A,$ is an indecomposable $A$-module if and only if $\Spec \hat{A}/0_d$ is
connected in codimension one, where $0_d$ denotes the intersection of the highest 
dimensional primary components of the zero ideal. See Corollary \ref{6.3} for 
the precise statements. One of our aims is to generalize their result.

The study of connectedness in commutative algebra and algebraic geometry has a 
long tradition. The first result in this direction was shown by Hartshorne 
(cf. \cite{rH}) who introduced the notion of connectedness in codimension one.

For a commutative ring $A$ of finite dimension we denote by $\mathbb G_A$ the undirected graph
whose vertices are primes ideals $\mathfrak p$ of $A$ with $\dim A = \dim A/\mathfrak p$ and two
distinct vertices $\mathfrak p, \mathfrak q$ are joined by an edge if and only if $\mathfrak p + \mathfrak q$ is of height one. For an ideal $I$ of $A$ let $I_d$ denote the intersection of all primary
components of $I$ that are of highest dimension.

\begin{theorem} \label{1.1} Let $(R, \mathfrak m)$ denote a Gorenstein ring. For 
an ideal $I$ of $R$ the following conditions are equivalent:
\begin{itemize}
  \item[(i)] The local cohomology module $H^c_I(R)$ is indecomposable.
  \item[(ii)] The graph $\mathbb G_{R/I}$ is connected.
  \item[(iii)] The variety $V(I_d)$ is connected in codimension one.
\end{itemize}
\end{theorem}

As a main technical tool we use connectedness properties of certain rings of sections. For a
local Noetherian ring $(A, \mathfrak m)$ and an ideal $I$ let $D_I(A)$ denote the ring of regular functions on $\Spec A \setminus V(I).$ Moreover, let $D^I(A) = \varprojlim D_{\mathfrak m}(A/I^{\alpha})$
denote the ring of regular formal functions of the completion of $U = \Spec A \setminus \{\mathfrak m\}$ along $V(I)\cap U.$  Then the following is shown:

\begin{theorem} \label{1.2} Let $I$ be an ideal in a local Noetherian ring $(A,\mathfrak m).$ The the
following conditions are equivalent:
\begin{itemize}
\item[(i)] $D^I(A)$ is indecomposable as a ring.
\item[(ii)] $V(I \hat A) \setminus \{\mathfrak m \hat A\}$ is connected.
\end{itemize}
\end{theorem}

This result has applications (cf. Theorem \ref{8.4}) related to the vanishing of local cohomology
$H^i_I(A)$ for $i = \dim A, \dim A -1.$ Connectedness and indecomposability are closely related as
it is shown in Section 2 (cf. in particular Proposition \ref{2.1}). In Section 3 we study the indecomposability of strict transforms in relation to connectedness. In Section 4 (cf. Lemma \ref{4.3}) we calculate the endomorphism ring of $H^c_I(R), c = \height I,$ for a one-dimensional ideal $I$ in a
Gorenstein ring $R.$ Note that recently there is some interest of endomorphism rings of this type (cf.
\cite{HeS} and the references there).

In Section 5 we modify the definition of a certain graph related to connectedness as introduced
by Hochster and Huneke (cf. \cite{HH}) and Lyubeznik (cf. \cite{gL2}). In Section 6 we use our methods in order to reprove the above mentioned result of Hochster and Huneke. The proof of Theorem \ref{1.1} is
done in Section 7. In the final Section 8 we prove Theorem \ref{1.2} and its consequences about
the formal cohomology. For some techniques of homological algebra we use the textbook \cite{cW}.

\section{Indecomposability}
Let $A$ denote a commutative Noetherian ring. A subset $T \subset \Spec A$ is called
stable with respect to specialization whenever $\mathfrak p \subset \mathfrak q$
and $\mathfrak p \in T$ implies that $\mathfrak q \in T.$ For instance
\[
V(I) = \{\mathfrak p \in \Spec A | I \subseteq \mathfrak p\},
\]
$I$ an ideal of $A,$ is stable with respect to specialization.

In the sequel there is a certain use of the indecomposability of an $A$-module $M.$ An $A$-module $M$ is called
indecomposable whenever a direct sum decomposition $M = M_1 \oplus M_2$ provides either $M_1
= 0$ or $M_2 = 0.$ Moreover $M$ is called indecomposable in
codimension one whenever $M_{\mathfrak p}$ is an indecomposable
$A_{\mathfrak p}$-module for all $\mathfrak p \in \Supp M$ with
$\dim M_{\mathfrak p} \leq 1.$

\begin{proposition} \label{2.1} Let $A$ be a commutative Noetherian
ring. Let $T$ be a subset of $\Spec A$ stable with  respect to specialization. Let $M$ be an
$A$-module satisfying the following conditions:
\begin{itemize}
\item[(a)] $\Ass_A M \cap T = \emptyset.$

\item[(b)] $\Supp_A M \setminus T$ is connected.

\item[(c)] $M_{\mathfrak p}$ is indecomposable (as $A_{\mathfrak p}$-module)
for all $\mathfrak p \in \Supp_A M \setminus T.$
\end{itemize}
Then $M$ itself is indecomposable as $A$-module.
\end{proposition}

\begin{proof} Suppose that $M = U \oplus V$ for two non-trivial
$A$-modules $U, V.$  Now choose two prime ideals  $\mathfrak p \in \Ass_A U$ and
$\mathfrak q \in \Ass_A V$ resp. Then
\[
\mathfrak p, \mathfrak q \in \Ass_A M \subset \Supp_A M \setminus T
\]
as it follows by virtue of (a).

By the assumption in (b) there is a family
\[
\mathfrak p = \mathfrak p_0, \mathfrak p_1, \ldots, \mathfrak p_s
= \mathfrak q
\]
of prime ideals in $\Supp_A M \setminus T$ and a family
$\mathfrak P_1, \ldots, \mathfrak P_s$
of prime ideals in $\Supp_A M \setminus T$ such that
\[
\mathfrak p_{i-1} + \mathfrak p_i \subset \mathfrak P_i, i =
1,\ldots,s.
\]
Because of $\mathfrak p = \mathfrak p_0 \in \Supp_A U$ it turns out that
$\mathfrak P_1 \in \Supp_A U.$ By the indecomposibility, as assumed by (c), we see that $\mathfrak P_1 \not\in \Supp_A V.$
Therefore $\mathfrak p_1 \not\in \Supp_A V$ while $\mathfrak p_1 \in \Supp_A U.$
Iterating this argument s times it follows that $\mathfrak q =
\mathfrak p_s \not\in \Supp_A V,$ in contradiction to the choice of
$\mathfrak q.$
\end{proof}

Next we present a first example of a module that is
indecomposable in codimension one.

\begin{proposition} \label{2.3} Let $(R, \mathfrak m)$ denote a
local Gorenstein ring. Let $I \subset R$ be an ideal with $c = \height I.$
The canonical module $K(R/I) = \Ext_R^c(R/I,R)$ is indecomposable
in codimension one.
\end{proposition}

\begin{proof}  Put $A = R/I.$ First note that $\Ass K(A) = \Ass R/I_d$ (cf. \cite{aG} resp. \cite[Lemma 1.9]{pS1}). Without loss of generality we may assume $I = I_d.$ Let $\mathfrak p \in \Supp K(A)$ denote a prime ideal with $t = \dim A_{\mathfrak p} \leq 1.$
Then $A_{\mathfrak p}$ is a Cohen-Macaulay ring of dimension $t$ and
$A_{\mathfrak p} \simeq K(K(A_{\mathfrak p}))$ (cf. for instance \cite[Theorem 1.14]{pS1}).
Moreover
\[
K(A_{\mathfrak p}) \simeq \Ext_{R_{\mathfrak p}}^c(A_{\mathfrak p},
R_{\mathfrak p}) \simeq K(A)\otimes_R R_{\mathfrak p},
\]
because of $c = \dim R_{\mathfrak p} - \dim A_{\mathfrak p}$ (cf.
\cite[Proposition 2.2]{pS2}).
Now suppose that $K(A_{\mathfrak p}) = U \oplus V$ is decomposable
with two non-trivial $A_{\mathfrak p}$-modules $U, V$ of dimension
$t.$ By passing to the canonical modules it provides a
decomposition
\[
A_{\mathfrak p} \simeq K(K(A_{\mathfrak p})) \simeq K(U) \oplus
K(V)
\]
of $A_{\mathfrak p}$ into two non-trivial $A_{\mathfrak
p}$-modules, which is absurd.
\end{proof}

Further examples of modules that are indecomposable in codimension one will be studied in the following sections.

\section{Ideal transforms}
Let $A$ denote a commutative ring and $I$ an ideal. For an $A$-module $M$ define
\[D_I(M) = \varinjlim \Hom_A(I^{\alpha}, M)\] the ideal transform of $M$ with respect to
$I.$ The homomorphisms in the direct system are induced by the inclusion. There are
a natural exact sequence
\[
0 \to H^0_I(M) \to M \to D_I(M) \to H^1_I(M) \to 0
\]
and isomorphisms $\operatorname{R}^iD_I(M) \simeq H^{i+1}_I(M)$ for all $i \geq 1.$
We refer to the book of Brodmann and Sharp (cf. \cite{BS}) for the details about this construction.

Recall that $D_I(A)$ admits the structure of a commutative ring, the ring of rational functions on $\Spec A \setminus V(I).$ Moreover, if $H^0_I(A) = 0,$ then
\[
D_I(A) = \{ q \in Q(A) | \Supp (Aq + A/A) \subset V(I) \},
\]
where $Q(A)$ denotes the full ring of quotients of $A.$ Next we summarize two additional
properties.

\begin{lemma} \label{3.1} Let $I, J$ be ideals of a commutative ring $A.$ Let $M$ denote an $A$-module. Then
\begin{itemize}
\item[(a)] There is an exact sequence, the Mayer-Vietoris sequence,
\[
0 \to D_{I + J}(M) \to D_I(M) \oplus D_J(M) \to D_{I \cap J}(M).
\]
\item[(b)] If $A$ is Noetherian, then $\Ass_A D_I(M) = \Ass_A M \setminus V(I).$
\end{itemize}
\end{lemma}

\begin{proof} The statement in (a) follows easily by the natural short exact sequences
\[
0 \to I^{\alpha} \cap J^{\alpha} \to I^{\alpha} \oplus J^{\alpha} \to I^{\alpha} + J^{\alpha} \to 0
\]
for all $\alpha \in \mathbb N$ (cf. \cite[Section 3]{BS} for the details).

For the proof of (b) let $\mathfrak p \in \Ass_A D_I(M).$ So there is an injection $A/\mathfrak p \to D_I(M).$ Then $\mathfrak p \not\in V(I)$ because any submodule
of $D_I(M)$ with support in $V(I)$ is zero. The reverse inclusion is easily seen by a localization argument.
\end{proof}

In the following we prove a result that seems to be well-known -- at least in algebraic geometry. We have not been able to find an appropriate reference in commutative algebra.

\begin{theorem} \label{3.2} Let $A$ denote a Noetherian ring. Let $I$ be a proper ideal of $A.$
Then the following conditions are equivalent:
\begin{itemize}
\item[(i)] $D_I(A)$ is indecomposable as an $A$-module.
\item[(ii)] $\Spec A \setminus V(I)$ is connected.
\end{itemize}
\end{theorem}

\begin{proof}
First we show the implication (i) $ \Longrightarrow $ (ii). Suppose that $\Spec A \setminus V(I)$ is disconnected. Then there are two proper ideals $I_i, i = 1,2,$ of $A$ such that
\begin{itemize}
\item $\Spec A \setminus V(I_i), i = 1,2,$ are disjoint and non-empty, and
\item $\Spec A \setminus V(I) = (\Spec A \setminus V(I_1)) \cup (\Spec A \setminus V(I_2)).$
\end{itemize}
Reading this two conditions in terms of the ideals it follows that  $\Rad I = \Rad (I_1 + I_2)$ and $I_1 \cap I_2$  is nilpotent. The Mayer-Vietoris sequence (cf.
Lemma \ref{3.1}) gives us an exact sequence
\[
0 \to D_I(A) \to D_{I_1}(A) \oplus D_{I_2}(A) \to D_{I_1 \cap I_2}(A) = 0.
\]
Now recall that $D_{I_1 \cap I_2}(A) = 0$ because $I_1 \cap I_2$ is nilpotent. Moreover
$D_{I_i}(A) \not= 0$ because $I_i, i = 1,2,$ is not nilpotent. So we have arrived at a contradiction.

For the prove of (ii) $ \Longrightarrow $ (i) first observe that we may assume $H^0_I(A) = 0$
as easily seen. For the proof we shall apply Proposition \ref{2.1}. In order to do so put $T = V(I).$ In order to complete the assumption we have to show that $D_I(A)\otimes A_{\mathfrak p}$ is indecomposable for all $\mathfrak p \in \Spec A \setminus V(I).$ Recall that
$\Supp_A D_I(A) = \Spec A \setminus V(I).$ But $D_I(A)\otimes A_{\mathfrak p} \simeq A_{\mathfrak p}$ for all $\mathfrak p \in \Spec A \setminus V(I).$
\end{proof}

As an application there is the following result originally shown by Hartshorne
(cf. \cite{rH}). We observe that $D_I(A) \simeq A$ if and only if $\grade I \geq 2.$

\begin{corollary} \label{3.3} Suppose that $\grade I \geq 2.$ Then $\Spec A \setminus V(I)$ is connected.
\end{corollary}

\section{Endomorphism rings of local cohomology}
Let $(A, \mathfrak m)$ denote a local ring. Let $I \subset A$ be an ideal and
$c = \grade I.$ In the following we need some technical results about the
endomorphism ring $\Hom_A(H^c_I(A), H^c_I(A)).$

\begin{lemma} \label{4.1} There are the following natural isomorphisms
\[
\Hom_A(H^c_I(A), H^c_I(A)) \simeq \Ext_A^c(H^c_I(A), A) \simeq \varprojlim
\Ext_A^c(\Ext_A^c(A/I^{\alpha}, A), A).
\]
If $A$ is a Gorenstein ring, then the endomorphism ring $\Hom_A(H^c_I(A), H^c_I(A))$ is commutative.
\end{lemma}

\begin{proof} By definition $H^c_I(A) \simeq \varinjlim \Ext_A^c(A/I^{\alpha}, A),$  so that
\[
 \Ext_A^c(H^c_I(A), A) \simeq \varprojlim \Ext_A^c(\Ext_A^c(A/I^{\alpha}, A), A).
\]
Now we recall that $\Ext_A^{\cdot}(\cdot,A)$ transforms a direct system into an inverse system. Let $A \qism E^{\cdot}$ denote a minimal injective resolution. Then there is an exact sequence $0 \to H^c_I(A) \to \gam(E^{\cdot})^c \to \gam(E^{\cdot})^{c+1}.$ To this end recall that $\gam(E^{\cdot})^i = 0$ for all $i < c.$ This is a consequence of the Matlis structure
theory about injective modules. It induces a natural commutative diagram with exact rows
\[
  \begin{array}{cccccc}
    0 \to & \Hom_A(H^c_I(A), H^c_I(A)) & \to & \Hom_A(H^c_I(A),\gam(E^{\cdot}))^c  & \to  & \Hom_A(H^c_I(A),\gam(E^{\cdot}))^{c+1} \\
      &     \downarrow             &      & \downarrow                          &      & \downarrow \\
    0 \to & \Ext_A^c(H^c_I(A), A)        & \to  & \Hom_A(H^c_I(A), E^{\cdot})^c       & \to  & \Hom_A(H^c_I(A), E^{\cdot})^{c+1} \\
  \end{array}
\]
because $\gam(E^{\cdot})$ is a subcomplex of $E^{\cdot}.$ The two last vertical homomorphisms
are isomorphisms. This follows because $\Hom_A(X, E_A(A/\mathfrak p)) = 0$ for an $A$-module $X$ with $\Supp_A X \subset V(I)$ and $\mathfrak p \not\in V(I).$ Therefore the first vertical map is also an isomorphism.

It is shown (cf. \cite[Theorem 3.1]{pS3} that $\Hom_A(H^c_I(A), H^c_I(A))$ is commutative
provided $A$ is a Gorenstein ring. In fact it is the inverse limit of commutative rings (cf. \ref{6.2}).
\end{proof}

\begin{remark} \label{4.2} While $\Hom_A(H^c_I(A), H^c_I(A))$ is a commutative ring provided
$A$ is a Gorenstein we do not know whether this is true in general. It is an open
problem whether the endomorphisms ring of $H^c_I(A), c = \height I,$ is a Noetherian ring
in the case $A$ a Gorenstein ring. It is a finitely generated $A$-module and
therefore Noetherian provided $(A, \mathfrak m)$ is a regular local ring containing a
field. See the discussion in Section 3 and 4 of the paper \cite{pS3}.
\end{remark}

In the following we shall give another partial affirmative answer to the question whether
the endomorphism ring of $H^c_I(R)$ is Noetherian.  For an ideal $I$ of a Noetherian ring $A$ define $u_A(I)$ the intersection of all
the $\mathfrak p$-primary components of the zero ideal with the property $\dim A/(I + \mathfrak p) > 0.$

\begin{lemma} \label{4.3} Let $(R, \mathfrak m)$ denote an $n$-dimensional Gorenstein ring. Let $I \subset R$ be an ideal such that $\dim R/I = 1.$ Then
\[
\hat R^I/u_{\hat R^I}(I) \simeq \Hom_R(H^c_I(R), H^c_I(R)),
\]
where $\hat R^I$ denotes the $I$-adic completion of $R.$  In particular, $\Hom_R(H^c_I(R), H^c_I(R))$ is a local Noetherian ring.
\end{lemma}

\begin{proof} Let $R \qism E^{\cdot}$ denote a minimal injective resolution of $R.$ We use the notion of the truncation complex $C_R^{\cdot}(I)$ (cf. \cite[Section 2]{HeS}). That is, there is an exact sequence
\[
0 \to H^c_I(R)[-c] \to \Gamma_I(E^{\cdot}) \to C_R^{\cdot}(I) \to 0,
\]
where $c = \height I.$ Because of $H^i(C_R^{\cdot}(I)) = 0$ for all $i \not= n$ there is an exact sequence
\[
0 \to \Ext^n_R(H^n_I(R), R) \to {\hat R}^I \to \Ext^c_R(H^c_I(R), R) \to \Ext^{n+1}_R(H^n_I(R), R) \to 0
\]
(cf. \cite[Lemma 2.4]{HeS}). Because $R$ is a Gorenstein ring $\operatorname{injdim} R = n$
so that the last $\Ext$-module in the exact sequence vanishes. Because of the Local Duality Theorem and Matlis duality there are the following isomorphisms
\[
\Ext^n_R(H^n_I(R), R) \simeq \varprojlim \Ext^n_R(\Ext^n_R(R/I^{\alpha}, R),R) \simeq \varprojlim H^0_{\mathfrak m}(R/I^{\alpha}).
\]
Finally $\varprojlim H^0_{\mathfrak m}(R/I^{\alpha}) = u_{\hat R^I}(I)$ follows by
\cite[Lemma 4.1]{pS4} with a slight modification of the arguments.
\end{proof}

\section{The connectedness graph}
For our purposes here we shall modify the definition of an undirected graph introduced
originally by Hochster and Huneke (cf. \cite[(3.4)]{HH}) and Lyubeznik (cf. \cite{gL2}).

\begin{definition} \label{5.1} Let $I$ denote an ideal of a Noetherian ring $A.$  Let $\mathbb  G(I)$ denote the following graph:
\begin{itemize}
\item The vertices of $\mathbb G(I)$ are the minimal prime ideals $\mathfrak p \in \Spec A \setminus V(I).$
\item Two vertices $\mathfrak p, \mathfrak q \in \Spec A \setminus V(I)$ are connected by an edge if there is a prime ideal $\mathfrak P \in \Spec A \setminus V(I)$ such that $\mathfrak p + \mathfrak q \subset \mathfrak P.$
\end{itemize}
\end{definition}

It is easily seen by the Definition \ref{5.1} that $\Spec A \setminus V(I)$ is connected if and
only if $\mathbb G(I)$ is connected. If $\mathbb G(I)$ is not connected let $\mathbb G_i, i = 1,\ldots, t,$ denote its connected components. Let $U_i, i = 1,\ldots, t$ denote the intersection of all $\mathfrak p$-primary components of the zero ideal of $A$ such that $\mathfrak p$ belongs to $\mathbb G_i.$ Clearly we have that $H^0_I(A) = \cap_{i=1}^t U_i.$

\begin{theorem} \label{5.2} Let $A$ be a Noetherian ring. Let $I$ denote an ideal of $A.$ With the previous notation there is a natural isomorphism
\[
D_I(A) \simeq \oplus_{i=1}^t D_I(A/U_i).
\]
Moreover $D_I(A/U_i), i = 1,\ldots, t,$ are indecomposable $A$-modules.
\end{theorem}

\begin{proof} We may assume $H^0_I(A) = 0$ without loss of generality. We put $U_0 = 0$ and proceed by an induction on $t$ in order to show the isomorphism of the statement. If $t = 1$ there is nothing to show. Let $U = \cap_{i = 2}^t U_i.$ Then $U_0 = U_1 \cap U.$ So, there is a natural short exact sequence
\[
0 \to A/U_0 \to A/U_1 \oplus A/U \to A/(U_1 + U) \to 0.
\]
By applying the section functor $D_I(\cdot)$ it induces an exact sequence
\[
0 \to D_I(A/U_0) \to D_I(A/U_1) \oplus D_I(A/U) \to D_I(A/(U_1 + U)).
\]
Now it is easily seen that $V(U_1 + U) = \cup_{i = 2}^t V(U_1 + U_i).$ For $i \not= 1$ the subgraphs $\mathbb G_1$ and $\mathbb G_i$ are not connected. Therefore any prime ideal $\mathfrak P \in V(U_1 + U_i), i = 2,\ldots, t,$ contains $I$ as follows by the definition. Therefore $V(U_1 + U) \subseteq V(I)$
and $D_I(A/U_1 + U)) = 0.$ This completes the inductive step.

Put $A_i = A/U_i, i = 1,\ldots,t.$ Then $\Spec A_i \setminus V(I A_i), i = 1,\ldots, t,$ is connected. Therefore, $D_I(A_i)$ is indecomposable (cf. Theorem \ref{3.2}).
\end{proof}

As a final point of this subsection we shall describe a situation when $D_I(A)$ is a local ring.

\begin{corollary} \label{5.3} Let $(A,\mathfrak m)$ denote a complete local ring. Suppose that
$\height (I + \mathfrak p)/\mathfrak p > 1$ for all $\mathfrak p \in \Ass A.$ Then $D_I(A)$ is a local ring if and only if $\Spec A \setminus V(I)$ is connected.
\end{corollary}

\begin{proof} By virtue of Theorem \ref{3.2} it will be enough to prove that -- under the additional assumptions -- the indecomposibility of $D_I(A)$ is equivalent to the fact that it is a local ring. By the Grothendieck finiteness result (cf. for instance \cite[Corollary 2.10]{pS1}) $D_I(A)$ is a finitely generated $A$-module. Because $A$ is a complete local ring and $D_I(A)$ is a finitely generated $A$-module the number of maximal ideals corresponds to the number of indecomposable components of $D_I(A)$
(cf. for instance \cite[Corollary 7.6]{dE}).
\end{proof}

\section{The case of maximal ideal}
In this section let $(A, \mathfrak m)$ denote a $d$-dimensional local ring. As
a first application of the above investigations we shall discuss whenever $H^d_{\mathfrak m}(A)$ is indecomposable. Therefore we shall recover the result of Hochster and Huneke
(cf. \cite{HH}). To this end we need the following definition of Hochster and Huneke (cf. \cite[(3.3)]{HH}) and Lyubeznik (cf. \cite{gL2}).

\begin{definition} \label{6.1} Let $A$ denote an
commutative Noetherian ring with finite dimension. We denote by
$\mathbb G_A$ the undirected graph whose vertices are primes
$\mathfrak p$ of $A$ such that $\dim A = \dim A/\mathfrak p,$ and
two distinct vertices $\mathfrak p, \mathfrak q$ are joined by
an edge if and only if $(\mathfrak p, \mathfrak q)$ is an ideal of
height one.
\end{definition}

We assume that $A$ is the factor ring of an $n$-dimensional Gorenstein ring $R.$ In the following we need a few properties about the $S_2$-fication of $A.$ For the detail we refer
to \cite{HH} and \cite{pS2}. Suppose that $A$ is equi dimensional. Then the $S_2$-fication $A \subset B$ of $A$ is given by
\[
B = \Hom_A(K(A),K(A)) \simeq \Ext_R^c(\Ext_R^c(A,R),R), \quad c = n -d,
\]
where $K(A) = \Ext_R^c(A,R)$ denotes the canonical module of $A.$

\begin{lemma} \label{6.2} Fix the previous notation. Suppose $A$ is equidimensional.
Let $C = B/A$ the cokernel of the inclusion $A \subset B.$ Define $I = \Ann_A C.$
\begin{itemize}
  \item[(a)] $V(I)$ is the non-$S_2$-locus of $A.$
  \item[(b)] $\dim A/I \leq \dim A - 2.$
  \item[(c)] $B \simeq D_I(A).$
\end{itemize}
\end{lemma}

\begin{proof} The statements in (a) and (b) are known (cf. for instance \cite[Lemma 5.3]{pS2}). For the proof of (c) we apply the functor $D_I(\cdot)$ to the short exact sequence $0 \to A \to B \to C \to 0.$ It provides an isomorphism $D_I(A) \simeq D_I(B).$ Moreover $H^i_I(B)= 0$ for $i \leq 1$ because $B$ satisfies $S_2$ and $\height I \geq 2.$ Therefore $\grade IB \geq 2$ and thus $B \simeq D_I(B).$
\end{proof}

As a consequence of the previous investigations we are able to prove the result of Hochster and Huneke (cf. \cite[(3.6)]{HH}).

\begin{corollary} \label{6.3} Let $(A, \mathfrak m)$ denote a $d$-dimensional complete local ring. Then the following conditions are equivalent:
\begin{itemize}
  \item[(i)] The local cohomology module $H^d_{\mathfrak m}(A)$ is indecomposable.
  \item[(ii)] The endomorphism ring $\Hom_A(H^d_{\mathfrak m}(A), H^d_{\mathfrak m}(A))$ is a local Noetherian ring, finitely generated as $A$-module.
  \item[(iii)] The graph $\mathbb G_A$ is connected.
  \item[(iv)] $V(0_d)$ is connected in codimension one.
\end{itemize}
\end{corollary}

\begin{proof} First of all note that we may reduce the statatements to the case of $0_d = 0.$ By the Cohen Structure Theorem we may assume that $A$ is the factor ring of an $n$-dimensional Gorenstein ring $R.$ 

By Matlis Duality it follows that $\Hom_A(H^d_{\mathfrak m}(A), H^d_{\mathfrak m}(A)) \simeq \Hom_A(K(A),K(A)).$ Moreover Matlis Duality provides also that $H^d_{\mathfrak m}(A)$ is indecomposable if and only if $K(A)$ is indecomposable. So (ii) implies (i) as easily seen. For the converse let $K(A)$ be an indecomposable $A$-module and assume that $B \simeq K(K(A)) = U\oplus V$ decomposes into two non-zero direct summands. Then $K(A) \simeq K(B) \simeq K(U) \oplus K(V)$ with two non-zero direct summands,
a contradiction.

By virtue of the definitions and Theorem \ref{3.2} as well as Corollary \ref{5.3} the last three conditions are equivalent to the fact that $D_I(A) \simeq \Hom_A(K(A),K(A))$ is indecomposable as $A$-module. Notice also that it is a finitely generated module over the complete local ring $A.$
\end{proof}

\section{The Gorenstein situation}
In this section we assume that $(R, \mathfrak m)$ is an $n$-dimensional Gorenstein ring. Let $I$ denote an ideal such that $c = \height I$ and $d = \dim R/I.$ We study the problem when $H^c_I(R)$ is indecomposable as an $R$-module.

\begin{theorem} \label{7.1} With the previous notation the following conditions are equivalent:
\begin{itemize}
  \item[(i)] The local cohomology module $H^c_I(R)$ is indecomposable.
  \item[(ii)] The graph $\mathbb G_{R/I}$ is connected.
  \item[(iii)] The variety $V(I_d)$ is connected in codimension one.
\end{itemize}
\end{theorem}

\begin{proof} First of all we shall prove that one might replace $I$ by $I_d.$ To this end we have to show that $H^c_I(R) \simeq H^c_{I_d}(R).$ So let $J$ denote the intersection of all the primary components of $I$ with dimension less than $d.$ Moreover, let $I_d = J_1\cap \ldots \cap J_s$ be a reduced primary decomposition and therefore $I = I_d \cap J.$ 

We follow an argument given by Lyubeznik (cf. \cite[Lemma 2.1]{gL2}). The Mayer-Vietoris sequence yields an exact sequence
\[
H^c_{I_d + J}(R) \to H^c_{I_d}(R) \oplus H^c_J(R) \to H^c_I(R) \to H^{c+1}_{I_d+J}(R).
\]
We have $\height J > c$ and thus $H^c_J(R) = 0.$ Moreover $\Rad (I_d + J) = \cap_{i=1}^s \Rad (J_i + J)$
and therefore $\height (I_d + J)  \geq c+ 2$ because of $\height J \geq c+1$ and $J_i \not\subseteq J.$ This
proves that $H^c_{I_d}(R) \simeq H^c_I(R).$

Now, the equivalence of (ii) and (iii) follows just by the definitions. Let us continue with the proof of the implication (i) $\Longrightarrow$ (ii).  Suppose that $\mathbb G_{R/I}$ is not connected
in codimension one. Let $\mathbb G_1, \ldots, \mathbb G_t, t > 1,$ denote the connected
components of $\mathbb G_{R/I}.$ Moreover, let $I_i, i = 1,\ldots,t,$ denote the intersection of all
minimal primes of $V(I)$ that are vertices of $\mathbb G_i.$ Then
\[
H^c_I(R) \simeq \oplus_{i = 1}^t H^c_{I_i}(R)
\]
as it is a consequence of the Mayer-Vietoris
sequence for local cohomology (cf.
\cite[Proposition 2.1]{gL2} for the details). Because of $H^c_{I_i}(R) \not= 0, i = 1,\ldots, t,$ this is a contradiction.

For the proof of (iii) $\Longrightarrow$ (i) we first prove that $H^c_I(R)$ is indecomposable
in codimension one. First we show $\Supp_R H^c_I(R) = V(I).$ Clearly $\Supp_R H^c_I(R) \subset V(I),$ so let $\mathfrak p \in V(I)$ be a minimal prime ideal. Then $H^c_{IR_{\mathfrak p}}(R_{\mathfrak p}) \simeq H^c_I(R) \otimes R_{\mathfrak p} \not= 0$ because $R_{\mathfrak p}$ is a Gorenstein ring. 

Now let $\mathfrak p \in \Supp_R H^c_I(R)$ with $\height \mathfrak p/I \leq 1.$ We have to show that $H^c_{I R_{\mathfrak p}}(R_{\mathfrak p})$ is indecomposable. For $\dim R_{\mathfrak p}/IR_{\mathfrak p} = 1$ this follows by Lemma \ref{4.3} while it is obviously true for $\dim R_{\mathfrak p}/IR_{\mathfrak p} = 0.$

Now suppose that $H^c_I(R) \simeq U \oplus V$ with $U, V \not= 0.$ Then
\[
\Ass_R R/I = \Ass_R H^c_I(R) = \Ass_R U \cup \Ass_R V.
\]
Choose prime ideals $\mathfrak p \in \Ass U$ and $\mathfrak q \in \Ass V.$ By the assumption there is a chain of prime ideals
\[
\mathfrak p = \mathfrak p_0, \mathfrak p_1, \ldots, \mathfrak p_s
= \mathfrak q
\]
of dimension $d$ in $\Supp_R H^c_I(R)$ and a family
$\mathfrak P_1, \ldots, \mathfrak P_s$
of prime ideals in $\Supp_R H^c_I(R)$ with the following properties
\[
\mathfrak p_{i-1} + \mathfrak p_i \subset \mathfrak P_i, i =
1,\ldots,s,
\]
and $\height \mathfrak P_i/I = 1, i = 1,\ldots, s.$ Now by the same screw driver argument as at the end of the proof of Theorem \ref{2.1} it follows that $\mathfrak q \not\in \Ass V,$ a contradiction.
\end{proof}

\begin{example} \label{7.2}
Let $k$ be a field and let
$A = k[|u,v,x,y|]$ be the formal power series ring in four
variables. Put $R = A/fA,$ where $f = xv-yu.$ Let $I = (x,y)R.$ It is
shown that
\[
B = \Hom_R(H^1_I(R), H^1_I(R))
\]
is not a finitely generated $R$-module, while it is a Noetherian ring (cf. \cite[Example 3.5]{pS3}).

To this end put $B_{\alpha} = k[|u,v|][a]/(a^{\alpha}),$ where $a$ is an indeterminate over $k[|u,v|].$Then there is a natural homomorphism $R \to B_{\alpha}$ induced by $x \mapsto ua, v \mapsto va.$ It turns out (cf. \cite[Example 3.5]{pS3}) that
\[
B \simeq \varprojlim B_{\alpha} \simeq k[|u,v,a|].
\]
Therefore $B$ is a Noetherian ring. Moreover $B$ is not a finitely generated $R$-module.
\end{example}

It is an open problem to us whether the conditions of Theorem \ref{7.1} are equivalent to the fact that $\Hom_R(H^c_I(R), H^c_I(R))$ is a local Noetherian ring. Moreover, we do not know how to generalize Theorem \ref{7.1} to the case of an arbitrary local ring. One of the critical points
is Lemma \ref{4.1}. We do not know a substitute for an arbitrary local ring.

\section{Formal cohomology}
Let $I$ denote an ideal of a local ring $(A, \mathfrak m).$ For an $A$-module $M$ the notion of the formal cohomology $\varprojlim H^i_{\mathfrak m}(M/I^{\alpha}M)$ was introduced in \cite{pS4}. There is an exact sequence
\[
0 \to \varprojlim H^0_{\mathfrak m}(M/I^{\alpha}M) \to {\hat M}^I \to D^I(M) \to \varprojlim
H^1_{\mathfrak m}(M/I^{\alpha}M) \to 0,
\]
where $D^I(M) = \varprojlim D_{\mathfrak m}(M/I^{\alpha}M).$ We refer to \cite{pS4} for basic definitions on formal cohomology modules. Clearly $D^I(A)$ admits the structure of a commutative ring. Here we start with a variation on the Mayer-Vietoris sequence on formal cohomology.

\begin{lemma} \label{8.1} Let $(A, \mathfrak m)$ denote a commutative ring. Let $M$ be a finitely generated $A$-module. There is an exact sequence
\[
0 \to D^{I\cap J}(M) \to D^I(M) \oplus D^J(M) \to D^{I + J}(M)
\]
for two ideals $I,J$ of $A.$
\end{lemma}

\begin{proof} For an arbitrary $\alpha \in \mathbb N$ there is a short exact sequence
\[
0 \to M/(I^{\alpha} \cap J^{\alpha})M \to M/I^{\alpha}M \oplus M/J^{\alpha}M \to M/(I^{\alpha} + J^{\alpha})M \to 0.
\]
Next we apply the ideal transform $D_{\mathfrak m}(\cdot)$ to this sequence. It induces an exact sequence
\[
0 \to D_{\mathfrak m}(M/(I^{\alpha} \cap J^{\alpha})M) \to D_{\mathfrak m}(M/I^{\alpha}M) \oplus D_{\mathfrak m}(M/J^{\alpha}M) \to D_{\mathfrak m}(M/(I^{\alpha} + J^{\alpha})M).
\]
The modules in the sequences form an inverse system. So we may pass to the inverse limit. Because the inverse limit is left exact, it provides the exact sequence of the claim. To do so we have to note the $I\cap J$-adic resp. the $(I + J)$-adic topology is equivalent to the topology defined by $(I^{\alpha} \cap J^{\alpha})M$ resp. $(I^{\alpha} + J^{\alpha})M$ (cf. \cite[Lemma 3.4, Lemma 3.8]{pS4}) for the details.
\end{proof}

As in the case of ideal transform we might use the exact sequence in \ref{8.1} in order to derive a connectedness results.

\begin{theorem} \label{8.2} Fix the previous notation and assume in addition that $A$ is a complete local ring. Then the following conditions are equivalent:
\begin{itemize}
  \item[(i)]  $D^I(A)$ is indecomposable as a ring.
  \item[(ii)] $V(I) \setminus \{\mathfrak m\}$ is connected.
\end{itemize}
\end{theorem}

\begin{proof} (i) $\Longrightarrow$ (ii): Suppose the contrary. That is $V(I) \setminus \{\mathfrak m\}$ is disconnected. Then there are two ideal $I_i, i = 1,2,$ such that the following conditions are satisfied:
\begin{itemize}
\item $V(I_i) \setminus \{\mathfrak m\} \not= \emptyset$ for $i = 1,2.$
\item $(V(I_1)\setminus \{\mathfrak m\}) \cup (V(I_2) \setminus \{\mathfrak m\}) = V(I) \setminus \{\mathfrak m\}.$
\item $(V(I_1)\setminus \{\mathfrak m\}) \cap (V(I_2) \setminus \{\mathfrak m\}) = \emptyset.$
\end{itemize}
This implies that $\Rad I_1 \cap I_2 = \Rad I$ and $I_1+I_2$ is an $\mathfrak m$-primary ideal. That is
\[
D^{I_1+I_2}(A) = \varprojlim D_{\mathfrak m}(A/(I_1+I_2)^{\alpha}) = 0
\]
as easily seen. Therefore $D^I(A) \simeq D^{I_1}(A) \oplus D^{I_2}(A)$ as it is a consequence  of the exact sequence shown in Lemma \ref{8.1}. In order to arrive at a contradiction we have to show that $D^{I_i}(A) \not= 0$ for $i = 1,2.$ Assume that $D^{I_i}(A) = 0.$ Then the exact sequence at the beginning of this section provides $\varprojlim H^0_{\mathfrak m}(A) = A$ because $A$ is a complete local ring. Therefore $\Supp A/I_i \subset \{\mathfrak m\}$ in contradiction to the first property above. Whence $D^{I_i}(A) \not= 0, i= 1,2,$ as required.

(ii) $\Longrightarrow$ (ii): As an inverse limit of commutative rings $D^I(A)$ admits the structure
of a commutative ring. Suppose that $D^I(A) = U \times V$ with two non-zero rings $U,V.$ So, there is a decomposition $1 = e_1 + e_2$ of the unit element into two non-trivial orthogonal idempotents $e_i, i = 1,2.$ Because of $D^I(A) \simeq \varprojlim D_{\mathfrak m}(A/I^{\alpha})$ the ring $D^I(A)$ is the inverse limit of the commutative rings $D_{\mathfrak m}(A/I^{\alpha}).$ Moreover, there are natural
homomorphisms of commutative rings $\phi_{\alpha} : D^I(A) \to D_{\mathfrak m}(A/I^{\alpha}), \alpha \in \mathbb N.$ Let $e_{i,\alpha} = \phi_{\alpha}(e_i)$ for $i = 1,2,$ and all  $\alpha \in \mathbb N.$

Because $V(I)\setminus \{\mathfrak m\}$ is connected it follows (cf. Theorem \ref{3.2}) that $D_{\mathfrak m}(A/I^{\alpha}), \alpha \in \mathbb N,$ is indecomposable. Therefore, there is an $i \in \{1,2\},$ say $i = 2,$ such that $e_{2,\alpha} = 0$ for infinitely many $\alpha \in \mathbb N.$  By the homomorphisms in the inverse system this implies that $e_{2,\alpha} = 0$ for all $\alpha \in \mathbb N.$
But this means that $1 = e_{1,\alpha}$ for all $\alpha \in \mathbb N$ and $e_2 = 0.$ But $V \simeq e_2 A = 0,$ in contradiction to $V \not= 0.$
\end{proof}

We do not know whether the statements of Theorem \ref{8.2} are equivalent to the fact that $D^I(A)$ is indecomposable as $A$-module. In the following we first derive a corollary of Theorem \ref{8.2} shown with different methods in
\cite[Lemma 5.4]{pS4}. It motivated the present considerations.

\begin{corollary} \label{8.3} Let $(A,\mathfrak m)$ denote a local ring. Suppose that $I \subset A$ is an ideal such that $\varprojlim H^i_{\mathfrak m}(A/I^{\alpha}) = 0$ for $i = 0, 1.$ Then $V(I \hat{A})
\setminus \{\mathfrak m\}$ is connected.
\end{corollary}

\begin{proof} It is easily seen that we may assume $(A, \mathfrak m)$ a complete local ring. Then by
view of the above exact sequence the assumption provides that $A \simeq D^I(A).$ That is, $D^I(A)$ is
indecomposable. By Theorem \ref{8.2} it follows that $V(I) \setminus \{\mathfrak m\}$ is connected.
\end{proof}

In the case of $(R, \mathfrak m)$ a Gorenstein ring it is known (see \cite{pS4}) by Matlis Duality that
\[
\varprojlim H^i_{\mathfrak m}(R/I^{\alpha}) \simeq \Hom_R(H^{n-i}_I(R), E).
\]
So the vanishing of the
formal cohomology $\varprojlim H^i_{\mathfrak m}(R/I^{\alpha}), i = 0, 1,$ is equivalent to the vanishing of $H^i_I(R), i = n, n - 1,$ where $n = \dim R.$ So we have that $V(I \hat R) \setminus \{\mathfrak m\}$ is connected whenever $H^i_I(R) = 0$ for $i = n, n-1.$ There are various extensions of
this observation to the case of arbitrary local rings (cf. for instance \cite[Theorem 5.6]{pS4}). 

\begin{lemma} \label{8.4} Let $(A, \mathfrak m)$ denote a local ring. Let $I \subset A$ denote an ideal such that $n = \dim A \geq 2.$ Suppose that:
\begin{itemize}
  \item[(a)] $\Ass \hat A = \Ass \hat A/0_d.$
  \item[(b)] $H^d_{\mathfrak m}(A)$ is indecomposable.
  \item[(c)] $H^i_I(A) = 0$ for $i \geq n-2.$
\end{itemize}
Then $V(I \hat A) \setminus \{\mathfrak m \hat A\}$ is connected. 
\end{lemma}

\begin{proof} Without loss of generality we may assume that $(A, \mathfrak m)$ is a complete local ring. Moreover, we may assume that $0_d = 0.$ By (a) and (b) it follows that 
\[
 A \subseteq B = \Hom_A(K(A), K(A)) \simeq \Hom_A(H^d_{\mathfrak m}(A), H^d_{\mathfrak m}(A))
\]
denotes the $S_2$-fication of $A/0_d$ (cf. Section 6). Then it was shown (cf. the proof of Theorem 5.6
in \cite{pS2}) that $\varprojlim H^i_{\mathfrak m}(B/I^{\alpha}B) = 0$ for $i = 0,1.$  
By virtue of \cite[Lemma 5.4]{pS2} it follows that 
\[
\Supp_A B/IB \setminus \{\mathfrak m\} = \Spec A/I \setminus \{\mathfrak m\}
\] 
is connected. 
\end{proof}

\end{document}